\documentclass{amsart}
\usepackage{amsmath,amsthm,amssymb,latexsym,verbatim,enumerate, mathrsfs,
enumitem, amsfonts, ifpdf}

\ifpdf
\usepackage{hyperref}
\else
\usepackage[hypertex]{hyperref}
\fi

\newcommand{\restrict}{\upharpoonright}

\newcommand{\bb}{\mathfrak{b}}
\newcommand{\dd}{{\mathfrak{d}}}

\newcommand{\inva}{{\mathfrak{a}}}

\renewcommand{\[}{\left[}
\renewcommand{\]}{\right]}

\newcommand{\PPP}{\mathcal{P}}

\newcommand{\lc}{\left|}
\newcommand{\rc}{\right|}

\newcommand{\<}{\prec}

\newcommand\CH{\mathrm{CH}}

\newcommand{\0}{\emptyset}
\DeclareMathOperator{\IniSeg}{IniSeg}

\DeclareMathOperator{\dom}{dom}

\DeclareMathOperator{\cf}{cf}

\newcommand{\Pset}{\mathcal{P}}

\newcommand{\CCC}{\mathcal{C}}

\newcommand{\pr}[2]{\left\langle #1, #2 \right\rangle}
\newcommand{\seq}[4]{\left\langle {#1}_{#2}: #2 #3 #4 \right\rangle}
\newcommand\w{\omega}

\newtheorem{thm}{Theorem}
\newtheorem{Claim}[thm]{Claim}
\newtheorem{cor}[thm]{Corollary}
\newtheorem{prop}[thm]{Proposition}
\newtheorem{lem}[thm]{Lemma}

\theoremstyle{definition}
\newtheorem{defn}[thm]{Definition}
\newtheorem{example}[thm]{Example}

\newtheorem{ques}[thm]{Question}

\numberwithin{thm}{section}
\usepackage{hyperref}
\begin{document}
\title[Order dimension]{On the order dimension of\\
locally countable partial orderings}
\author[Higuchi]{Kojiro Higuchi}
\address[Higuchi]{College of Engineering\\
  Nihon University\\
  1 Nakagawara, Tokusada, Tamuramachi, Koriyama\\
  Fukushima Prefecture 963-8642\\
  JAPAN}
\email{\href{mailto:higuchi.koujirou@nihon-u.ac.jp}%
  {higuchi.koujirou@nihon-u.ac.jp}}
\urladdr{\url{http://kenkyu-web.cin.nihon-u.ac.jp/%
Profiles/130/0012920/prof_e.html}}
\author[Lempp]{Steffen Lempp}
\address[Lempp]{Department of Mathematics\\
  University of Wisconsin\\
  Madison, Wisconsin 53706-1325\\
  USA}
\email{\href{mailto:lempp@math.wisc.edu}{lempp@math.wisc.edu}}
\urladdr{\url{http://www.math.wisc.edu/~lempp}}
\author[Raghavan]{Dilip Raghavan}
\address[Raghavan, Stephan]{Department of Mathematics\\
  National University of Singapore\\
  10 Lower Kent Ridge Road\\
  Singapore 119076, Republic of Singapore}
\email{\href{mailto:matrd@nus.edu.sg}{matrd@nus.edu.sg}}
\urladdr{\url{http://www.math.nus.edu.sg/~raghavan/}}
\author[Stephan]{Frank Stephan}
\email{\href{mailto:fstephan@comp.nus.edu.sg}{fstephan@comp.nus.edu.sg}}
\urladdr{\url{http://www.comp.nus.edu.sg/~fstephan/}}
\thanks{The first author was partially supported by Grant-in-Aid for JSPS
Fellows. The second author was partially supported by NSF grant DMS-160022;
he also wishes to thank the National University of Singapore for its
hospitality during his sabbatical year when most of this research was carried
out. The third author was partially supported by the Singapore Ministry of
Education's research grant number MOE2017-T2-2-125. The fourth author was
partially supported in part by the Singapore Ministry of Education Academic
Research Fund Tier 2 grant MOE2016-T2-1-019 / R146-000-234-112. }
\begin{abstract}
We show that the order dimension of the partial order of all finite subsets
of~$\kappa$ under set inclusion is~${\log}_{2}({\log}_{2}(\kappa))$
whenever~$\kappa$ is an infinite cardinal.

We also show that the order dimension of any locally countable partial
ordering~$(P, <)$ of size~$\kappa^+$, for any $\kappa$ of uncountable
cofinality, is at most~$\kappa$. In particular, this implies that it is
consistent with ZFC that the dimension of the Turing degrees under partial
ordering can be strictly less than the continuum.
\end{abstract}
\maketitle

\section{Introduction}

\noindent
This paper arose from a question posed by the first to the third author at
the Computability Theory and the Foundations of Mathematics conference at
Tokyo in 2015 regarding a set-theoretic question about a
computability-theoretic structure:

\begin{ques}[Higuchi]
What is the order dimension of the Turing degrees regarded as a partial
order?
\end{ques}

\noindent
Higuchi had already shown that this dimension must be uncountable and asked
whether it is the continuum. This paper provides a partial answer: It is
consistent with ZFC that the dimension is less than the continuum. But
Higuchi's question raised a number of related questions to which we give some
answers in this paper, all about the order dimension of locally finite and
locally countable partial orders.

We start with some definitions:

\begin{defn}[Dushnik, Miller~\cite{DushMill}, Ore~\cite{Ore}]
Given a partial order $\PPP = (P, \<)$, the \emph{order dimension} (or simply
\emph{dimension}) of~$\PPP$ is the smallest cardinality of a collection of
linearizations of~$\<$ which intersect to~$\<$.
\end{defn}

\noindent
So, for example, the dimension of a linear order is clearly~1, and the
dimension of an antichain is easily seen to be~2. It is also easy to see that
the dimension of an infinite partial order~$\PPP$ can be at most~$|P|$: For
each pair $x, y \in P$ with $y \not\preceq x$, fix a linearization~$<_{x,y}$
of~$\<$ with $x <_{x,y} y$.

\begin{defn}
Call a partial order $\PPP = (P, \<)$ \emph{locally finite} (\emph{locally
countable}, respectively), if for each $x \in P$, the set $\{y \in P \mid y
\< x\}$ is finite or countable, respectively.
\end{defn}

\noindent
Partial orders which are locally finite (or locally countable, respectively)
are also often said to have the \emph{finite predecessor property} (or the
\emph{countable predecessor property}).

The order dimension of the Turing degrees can be thought of as a new cardinal
invariant because it is between~${\aleph}_{1}$ and~${2}^{{\aleph}_{0}}$.

\begin{defn} \label{def:diminv}
Let $\pr{\mathcal{D}}{{<}_{T}}$ denote the class of Turing degrees equipped
with the ordering of Turing reducibility. The cardinal ${\mathfrak{dim}}_{T}$
denotes the order dimension of $\pr{\mathcal{D}}{{<}_{T}}$.
\end{defn}

\noindent
Since $\mathcal{D}$ has cardinality ${2}^{{\aleph}_{0}}$,
${\mathfrak{dim}}_{T} \leq {2}^{{\aleph}_{0}}$, and by Higuchi's
Proposition~\ref{prop:higuchi}, ${\aleph}_{1} \leq {\mathfrak{dim}}_{T}$.
Thus the cardinal ${\mathfrak{dim}}_{T}$ sits between~${\aleph}_{1}$
and~${2}^{{\aleph}_{0}}$, like many of the standard cardinal invariants of
the continuum such as~$\bb$, $\dd$,~$\inva$, etc. The reader is referred to
Blass~\cite{blasssmall} for a general survey of combinatorial cardinal
characteristics of the continuum. Of course, under $\CH$,
${\mathfrak{dim}}_{T} = {\aleph}_{1} = {2}^{{\aleph}_{0}}$. In this paper, we
will show that ${\mathfrak{dim}}_{T}$ is smaller than ${2}^{{\aleph}_{0}}$
``most of the time''. More precisely, we will show that there are only three
circumstances under which ${\mathfrak{dim}}_{T} = {2}^{{\aleph}_{0}}$ is
possibly consistent: ${2}^{{\aleph}_{0}} = {\aleph}_{1}$, or
${2}^{{\aleph}_{0}}$ is a limit cardinal (either singular or weakly
inaccessible), or ${2}^{{\aleph}_{0}}$ is the successor of a singular
cardinal of countable cofinality.

We will now present some results on the dimension of locally finite partial
orders in Section~\ref{sec:locfin}, and on the dimension of locally countable
partial orders in Section~\ref{sec:loccoun}, and close with some results on
the dimension of degree structures from computability theory in
Section~\ref{sec:degstr}.

\section{\texorpdfstring{The dimension of ${\[\kappa\]}^{<\omega}$}%
{The dimension of the collection of finite subsets of kappa}}%
\label{sec:locfin}

\noindent
We determine the order dimension of $\pr{{\[\kappa\]}^{<
\omega}}{\subsetneq}$, which is universal among locally finite posets of
cardinality~$\kappa$, i.e., every locally finite poset $\PPP = (P, \<)$ with
$|P|=\kappa$ embeds into $\pr{{\[\kappa\]}^{< \omega}}{\subsetneq}$ by
assigning $a\in P$ to $\{f(b):b\in P, b\preceq a\}$, where $f:P\to \kappa$ is
any injection. Thus our result in this section provides an upper bound for
every locally finite poset.

\begin{defn} \label{def:log}
For any infinite cardinal~$\kappa$,
\begin{align*}
{\log}_{2}(\kappa) = \min\{\lambda: {2}^{\lambda} \geq \kappa\}
\end{align*}
\end{defn}

\noindent
The main theorem of this section is the following.

\begin{thm} \label{thm:doublelog}
Let $\kappa \geq \omega$ be any cardinal. Then the order dimension of
$\pr{{\[\kappa\]}^{< \omega}}{\subsetneq}$ is
${\log}_{2}({\log}_{2}(\kappa))$.
\end{thm}

\noindent
It follows from Definition~\ref{def:log} that
${\log}_{2}({\log}_{2}(\kappa))$ is the minimal~$\lambda$ such that
${2}^{{2}^{\lambda}} \geq \kappa$. Theorem~\ref{thm:doublelog} says in
particular that $\pr{{\[{\omega}_{1}\]}^{< \omega}}{\subsetneq}$ and indeed
$\pr{{\[{2}^{{2}^{\omega}}\]}^{< \omega}}{\subsetneq}$ have countable
dimension. But $\pr{{\[{\left({2}^{{2}^{\omega}} \right)}^{+}\]}^{<
\omega}}{\subsetneq}$ has uncountable dimension. Theorem~\ref{thm:doublelog}
will be proved via two lemmas that establish upper and lower bounds on the
dimension of $\pr{{\[\kappa\]}^{< \omega}}{\subsetneq}$. The lower bound is
established using a result from partition calculus. Recall the following
consequence of the well-known Erd{\H o}s-Rado theorem.

\begin{thm}[Corollary 17.11 in~\cite{partitionbible}] \label{thm:erdosrado}
For any cardinal $\rho$
and any nonzero $r < \omega$,
\begin{align*}
{\left({\exp}_{r - 1}(\rho)\right)}^{+} \rightarrow
   {\left({\rho}^{+}\right)}^{r}_{\rho}.
\end{align*}
\end{thm}

\noindent
The ${\exp}_{i}(\rho)$ operation is defined by recursion on $i < \omega$:
${\exp}_{0}(\rho) = \rho$; ${\exp}_{i + 1}(\rho) = {2}^{{\exp}_{i}(\rho)}$.
Setting $r = 3$ in Theorem~\ref{thm:erdosrado} gives ${\left(
{2}^{{2}^{\rho}} \right)}^{+} \rightarrow {\left( {\rho}^{+}
\right)}^{3}_{\rho}$.

\begin{lem} \label{lem:lowerbound}
Let $\kappa \geq \omega$ be a cardinal. Then the order dimension of
$\pr{{\[\kappa\]}^{< \omega}}{\subsetneq}$ is at least
${\log}_{2}({\log}_{2}(\kappa))$.
\end{lem}

\begin{proof}
Let~$\rho$ denote the order dimension of $\pr{{\[\kappa\]}^{<
\omega}}{\subsetneq}$. Suppose for a contradiction that $\rho <
{\log}_{2}({\log}_{2}(\kappa))$. In other words, ${2}^{{2}^{\rho}} < \kappa$,
and so $\theta = \max\{\omega,{\left( {2}^{{2}^{\rho}} \right)}^{+}\} \leq
\kappa$. Let $\seq{<}{\iota}{<}{\rho}$ be a sequence of linear orders
witnessing that $\pr{{\[\kappa\]}^{< \omega}}{\subsetneq}$ has
dimension~$\rho$. Thus for all $X, Y \in {\[\kappa\]}^{< \omega}$, if $X
\subsetneq Y$, then $\forall \iota < \rho\[X \: {<}_{\iota} \: Y\]$, while if
$X \not\subseteq Y$, then $\exists \iota < \rho\[Y \: {<}_{\iota} \: X\]$.
Now define a coloring $c: {\[\theta\]}^{3} \rightarrow \rho$ as follows.
Given any $\{\alpha, \beta, \gamma\} \in {\[\theta\]}^{3}$ with $\alpha <
\beta < \gamma$, there exists $\iota < \rho$ such that $\{\alpha, \gamma\} \:
{<}_{\iota} \: \{\beta\}$ because $\{\alpha, \gamma\}, \{\beta\} \in
{\[\kappa\]}^{< \omega}$ and $\{\beta\} \not\subseteq \{\alpha, \gamma\}$.
Define $c(\{\alpha, \beta, \gamma\}) = \min\{\iota < \rho: \{\alpha, \gamma\}
\: {<}_{\iota} \: \{\beta\} \}$. Using Ramsey's Theorem when $\rho$ is finite
and Theorem~\ref{thm:erdosrado} otherwise, fix $\alpha < \beta < \gamma <
\delta < \theta \leq \kappa$ and $\iota < \rho$ so that $\{\alpha, \gamma\}
\: {<}_{\iota} \: \{\beta\}$ and $\{\beta, \delta\} \: {<}_{\iota} \:
\{\gamma\}$. However, we now have $\{\gamma\} \: {<}_{\iota} \: \{\alpha,
\gamma\} \: {<}_{\iota} \: \{\beta\} \: {<}_{\iota} \: \{\beta, \delta\} \:
{<}_{\iota} \: \{\gamma\}$, which is impossible. This contradiction completes
the proof of the lemma.
\end{proof}

\noindent
Note that the proof of Lemma~\ref{lem:lowerbound} shows that
$\pr{{\[\kappa\]}^{\leq 2}}{\subsetneq}$ has order dimension at least
${\log}_{2}({\log}_{2}(\kappa))$. Also the proof only needed the partition
relation ${\left( {2}^{{2}^{\rho}} \right)}^{+} \rightarrow {\left( 4
\right)}^{3}_{\rho}$, and $\omega \rightarrow {\left( 4
\right)}^{3}_{\rho}$ when $\rho<\omega$ (actually something even weaker; just
a monochromatic shift is needed).

\begin{defn} \label{def:indep}
Let $\theta \geq \omega$ be a cardinal. For $X \subseteq \theta$, we use
${X}^{0}$ to denote $X$ and ${X}^{1}$ to denote $\theta \setminus X$. We say
that a sequence $\seq{X}{i}{\in}{I}$ of subsets of~$\theta$ is
\emph{independent} if for each nonempty finite set $F \subseteq I$ and for
each function $\sigma: F \rightarrow 2$, the set ${\bigcap}_{i \in
F}{{X}^{\sigma(i)}_{i}}$ has cardinality~$\theta$.
\end{defn}

\begin{thm}[Hausdorff; see Exercise (A6) in Chapter VIII of~\cite{Kunen}]%
\label{thm:indep} For each cardinal $\theta \geq \omega$, there is an
independent family $\seq{X}{\alpha}{<}{{2}^{\theta}}$ of subsets of~$\theta$.
\end{thm}

\begin{lem} \label{lem:upperbound}
Let $\kappa \geq \omega$ be a cardinal. Then the order dimension of
$\pr{{\[\kappa\]}^{< \omega}}{\subsetneq}$ is at most
${\log}_{2}({\log}_{2}(\kappa))$.
\end{lem}

\begin{proof}
Let $\theta = {\log}_{2}({\log}_{2}(\kappa))$, and let $\lambda =
{2}^{\theta}$. Then ${2}^{\lambda} \geq \kappa$, and so it is possible to
find a sequence $\seq{f}{\alpha}{<}{\kappa}$ of~$\kappa$ many distinct
functions from~$\lambda$ to~$2$. To avoid confusion, we use ${}^{\lambda}{2}$
to denote the collection of all functions from~$\lambda$ to~$2$. Next, use
Theorem~\ref{thm:indep} to fix an independent family
$\seq{X}{\xi}{<}{\lambda}$ of subsets of~$\theta$. For each $\iota < \theta$,
define a function ${g}_{\iota}: \lambda \rightarrow 2$ by stipulating that
${g}_{\iota}(\xi) = 0$ iff $\iota \in {X}_{\xi}$. Given $f, g \in
{}^{\lambda}{2}$ with $f \neq g$, let $\Delta(f, g) = \min\{\xi < \lambda:
f(\xi) \neq g(\xi)\}$. Unfixing $\iota < \theta$, we are going to define a
linear order ${<}_{\iota}$ on ${\[\kappa\]}^{< \omega}$ for each $\iota <
\theta$. First declare, for each $\iota < \theta$ and all $u, v \in
{\[\kappa\]}^{< \omega}$, that $u \: {\<}_{\iota} \: v$ if and only if either
$u \subsetneq v$ or~$f_\alpha$ coincides with~$g_\iota$ on
$\Delta(u,f_\alpha)=\{\Delta(f_\beta,f_\alpha): \beta\in u\}$ for some
$\alpha\in v\setminus u$, i.e., $f_\alpha(\xi)=g_\iota(\xi)$ for each $\xi\in
\Delta(u,f_\alpha)$.

\begin{Claim} \label{claim:upperbound}
For each $\iota < \theta$,~${\<}_{\iota}$ is a partial order on
${\[\kappa\]}^{< \omega}$.
\end{Claim}

\begin{proof}
Clearly, $u \: {\nprec}_{\iota}\: u$ because both the defining clauses of
${\<}_{\iota}$ fail in this case. So it remains to check that~${\<}_{\iota}$
is transitive. Consider $u, v, w \in {\[\kappa\]}^{< \omega}$ and suppose
that $u \: {\<}_{\iota} \: v$ and that $v \: {\<}_{\iota} \: w$. There are
four cases to consider.

Case~1: $u \subsetneq v$ and $v \subsetneq w$. Then $u \subsetneq w$, whence
$u \: {\<}_{\iota} \: w$.

Case~2: $u \subsetneq v$ and~$f_\alpha$ coincides with~$g_\iota$ on
$\Delta(v,f_\alpha)$ for some $\alpha\in w \setminus v$. Then, such~$\alpha$
is an element of $w \setminus u$, and~$f_\alpha$ coincides with~$g_\iota$ on
$\Delta(u,f_\alpha)$ since $\Delta(u,f_\alpha)\subset\Delta(v,f_\alpha)$ by
$u\subsetneq v$. So $u \: {\<}_{\iota}\: w$.

Case~3:~$f_\alpha$ coincides with~$g_\iota$ on $\Delta(u,f_\alpha)$ for some
$\alpha\in u\setminus v$ and $v \subsetneq w$. Then, such $\alpha\in
w\setminus u$, and~$f_\alpha$ coincides with~$g_\iota$ on
$\Delta(u,f_\alpha)$. So $u \: {\<}_{\iota}\: w$.

Case~4:~$f_\beta$ coincides with~$g_\iota$ on $\Delta(u,f_\beta)$ for some
$\beta\in v\setminus u$, and~$f_\gamma$ coincides with~$g_\iota$ on
$\Delta(v,f_\gamma)$ for some $\gamma\in w\setminus v$ . Note that~$f_\beta$
and~$f_\gamma$ coincide with~$g_\iota$ on $\Delta(u,f_\beta)$ and
$\Delta(v,f_\gamma)$, respectively, and $f_\beta(\Delta(f_\beta,f_\gamma))\ne
f_\gamma(\Delta(f_\beta,f_\gamma))=g_\iota(\Delta(f_\beta,f_\gamma))$. Hence
we have $\Delta(f_\beta,f_\gamma)\not\in\Delta(u,f_\beta)$, which implies
$\gamma \not\in u$ and so $\gamma \in w \setminus u$. Fixing $\alpha\in
u$, we next show that $f_\gamma(\Delta(f_\alpha,f_\gamma)) =
g_\iota(\Delta(f_\alpha,f_\gamma))$. From the above conclusion, we know that
$\Delta(f_\alpha,f_\beta)\ne\Delta(f_\beta,f_\gamma)$. There are two subcases
to check.

Subcase~4.1: $\Delta(f_\alpha,f_\beta)<\Delta(f_\beta,f_\gamma)$. Then
$\Delta(f_\alpha,f_\gamma) = \Delta(f_\alpha,f_\beta)$, and therefore
$f_\gamma(\Delta(f_\alpha,f_\gamma))=f_\gamma(\Delta(f_\alpha,f_\beta)) =
f_\beta(\Delta(f_\alpha,f_\beta))=g_\iota(\Delta(f_\alpha,f_\beta)) =
g_\iota(\Delta(f_\alpha,f_\gamma))$.

Subcase~4.2: $\Delta(f_\alpha,f_\beta)>\Delta(f_\beta,f_\gamma)$. Then
$\Delta(f_\alpha,f_\gamma)=\Delta(f_\beta,f_\gamma)$, and therefore
$f_\gamma(\Delta(f_\alpha,f_\gamma))=f_\gamma(\Delta(f_\beta,f_\gamma))
=g_\iota(\Delta(f_\beta,f_\gamma))=g_\iota(\Delta(f_\alpha,f_\gamma))$.

We conclude that $\gamma \in  w \setminus u$ and that ${f}_{\gamma}$
coincides with~$g_\iota$ on $\Delta(u,f_\gamma)$. So $u \: {\<}_{\iota}\: w$.

This completes the proof of the claim.
\end{proof}

\noindent
Now, for each $\iota < \theta$, let~${<}_{\iota}$ be any linear order on
${\[\kappa\]}^{< \omega}$ extending~${\<}_{\iota}$. Note that for any $u, v
\in {\[\kappa\]}^{<\omega}$, if $u \subsetneq v$, then $u \: {\<}_{\iota} \:
v$, and hence $u \:{<}_{\iota} \: v$. Thus $\seq{<}{\iota}{<}{\theta}$ is a
sequence of linear orders on ${\[\kappa\]}^{< \omega}$ compatible
with~$\subsetneq$. To see that this sequence of linear orders witnesses that
the dimension of ${\[\kappa\]}^{< \omega}$ is at most~$\theta$, fix any $u, v
\in {\[\kappa\]}^{< \omega}$ and suppose that $v \not\subseteq u$. It
suffices to find an $\iota < \theta$ such that $u \: {\<}_{\iota} \: v$. Fix
$\alpha \in v \setminus u$. If $u = \0$, then $\0 \subsetneq v$, and so for
every $\iota < \theta$, $u \: {\<}_{\iota} \: v$. Therefore we may assume
that $u \neq\0$. Define $F = \Delta(u,f_\alpha)$.~$F$ is a nonempty finite
subset of~$\lambda$. Define $\sigma: F \rightarrow 2$ by setting $\sigma(\xi)
= {f}_{\alpha}(\xi)$, for all $\xi \in F$. Since $\seq{X}{\xi}{<}{\lambda}$
is an independent family of subsets of~$\theta$, ${\bigcap}_{\xi \in
F}{{X}^{\sigma(\xi)}_{\xi}}$ is nonempty. Let $\iota = \min\left(
{\bigcap}_{\xi \in F}{{X}^{\sigma(\xi)}_{\xi}} \right)$. Then $\iota \in
\theta$ and for each $\beta \in u$, ${g}_{\iota}(\Delta({f}_{\alpha},
{f}_{\beta})) = \sigma(\Delta({f}_{\alpha}, {f}_{\beta})) =
{f}_{\alpha}(\Delta({f}_{\alpha}, {f}_{\beta}))$. In other words,~$f_\alpha$
coincides with~$g_\iota$ on $\Delta(u,f_\alpha)$. So $u \: {\<}_{\iota} \: v$
as claimed.
\end{proof}

\section{The dimension of locally countable partial orderings}%
\label{sec:loccoun}

\noindent
The setting of locally countable partial orders (for which the Turing degrees
and many other degree structures from computability theory form natural
examples) is quite a bit more complicated.

Even though the following theorem is not needed for the proof of our main
result, it provides information about arbitrary locally countable partial
orders which is likely to be useful in their analysis. It proves the
existence of a ranking function on them.
\begin{lem}\label{lem:rank}
Suppose $\PPP = (P,\<)$ is any locally countable partial ordering. Then there
is a function $r: P \to \eta \cdot \w_1$ such that for all $x, y \in P$, $x
\< y$ implies $r(x)<r(y)$. (Here, $\eta \cdot \w_1$ is the order product of
these two order types under the antilexicographical ordering.)
\end{lem}

\begin{proof}
If $P$ is empty, then there is nothing to prove. So we assume that $P$ is
non-empty. Let $\kappa = \lc P \rc$. From an enumeration $\langle y_\beta:
\beta<\kappa \rangle$ of~$P$, we construct a cofinal sequence $\langle
x_\alpha: \alpha < \lambda \rangle$ of elements of~$P$ by recursion as
follows: Let $x_0 = y_0$, and for $\alpha > 0$, let $x_\alpha = y_\beta$ for
the least~$\beta$ such that $y_\beta \not\preceq x_{\alpha'}$ for any
$\alpha' < \alpha$. (The recursion stops at the least ordinal~$\lambda$ when
there is no such~$\beta$.)

Now, for each $\alpha < \lambda$, let $A_\alpha = \{x \in P: x \preceq
x_\alpha\}$ and recall that by local countability, each~$A_\alpha$ is
countable. Now define~$r$ by recursion on $\alpha < \lambda$: Let $r
\restrict A_0$ map~$A_0$ into $\eta \cdot \{0\}$ using any linearization of
$\< \restrict A_0^2$. For~$\alpha$ with $0 < \alpha < \lambda$, assume that
$r\restrict\left( {\bigcup}_{\alpha' < \alpha}{{A}_{\alpha'}} \right)$ has
been defined. Find a countable set $B_\alpha \subseteq \alpha$ such that
\begin{align*}
  A_\alpha \cap \bigcup_{\alpha' < \alpha} A_{\alpha'} =
  A_\alpha \cap \bigcup_{\alpha' \in B_\alpha} A_{\alpha'}.
\end{align*}
Fix $\gamma < \w_1$ such that
\begin{align*}
  r''\left(\bigcup_{\alpha' \in B_\alpha} A_{\alpha'}\right) \subseteq \eta
       \cdot \gamma.
\end{align*}
Now extend the definition of~$r$ to $A_\alpha \setminus \left(
\bigcup_{\alpha' \in B_\alpha} A_{\alpha'} \right)$ by mapping this set into
$\eta \cdot \{\gamma\}$ using any linearization of~$\<$ on this set. It is
clear that such countable~$\gamma$ must exist because $\bigcup_{\alpha' \in
B_\alpha} A_{\alpha'}$ is countable, giving us the desired map~$r$.
\end{proof}

\noindent
Our main theorem on locally countable partial orders is the following:

\begin{thm}\label{poset-dim}
Let~$\kappa$ be a regular uncountable cardinal and $\PPP = (P,\<)$ a locally
countable partial order of size~$\kappa^+$. Then~$\PPP$ has dimension at
most~$\kappa$.
\end{thm}

\noindent
Our proof of this theorem uses two lemmas:

\begin{lem}\label{lem:func}
For~$\kappa$ and $\PPP = (P,\<)$ as in Theorem~\ref{poset-dim}, there is a
sequence $\langle f_\alpha: \alpha < \kappa \rangle$ of functions $f_\alpha:
P \to 2$ such that for any countable subset $A \subseteq P$ and any $y \in P
\setminus A$, there is $\alpha < \kappa$ such that $f_\alpha''A = \{0\}$ and
$f_\alpha(y)=1$.
\end{lem}

\begin{proof}
The hypothesis is that $\lc P \rc = {\kappa}^{+}$. Let
$\seq{x}{\xi}{<}{{\kappa}^{+}}$ be a 1--1 enumeration of~$P$. The proof will
proceed via three claims.

\begin{Claim} \label{claim:reading}
Suppose that there is a sequence $\seq{g}{\xi}{<}{{\kappa}^{+}}$ such that
for each $\xi < {\kappa}^{+}$, ${g}_{\xi}: \kappa \rightarrow 2$, and that
for every countable $B \subseteq {\kappa}^{+}$ and every $\zeta \in
{\kappa}^{+} \setminus B$, there exists $\alpha < \kappa$ so that $\forall
\xi \in B\[{g}_{\xi}(\alpha) = 0\]$ and ${g}_{\zeta}(\alpha) = 1$. Then there
exists a sequence of functions $\seq{f}{\alpha}{<}{\kappa}$ as in the
statement of the lemma.
\end{Claim}

\begin{proof}
Let $\seq{g}{\xi}{<}{{\kappa}^{+}}$ be given. Fix $\alpha < \kappa$ and
define ${f}_{\alpha}: P \rightarrow 2$ by stipulating that
${f}_{\alpha}({x}_{\xi}) = {g}_{\xi}(\alpha)$ for each $\xi < {\kappa}^{+}$.
Suppose that $A \subseteq P$ is countable and that $y \in P \setminus A$. Let
$B = \{\xi < {\kappa}^{+}: {x}_{\xi} \in A\}$. Then~$B$ is a countable subset
of ${\kappa}^{+}$ because we have a 1--1 enumeration of~$P$. Let $\zeta <
{\kappa}^{+}$ be so that $y = {x}_{\zeta}$. Note $\zeta \in {\kappa}^{+}
\setminus B$. So there is $\alpha < \kappa$ so that $\forall \xi \in
B\[{g}_{\xi}(\alpha) = 0\]$ and ${g}_{\zeta}(\alpha) = 1$. If $x \in A$, then
$x = {x}_{\xi}$ for some $\xi \in B$, and so ${f}_{\alpha}(x) =
{f}_{\alpha}({x}_{\xi}) = {g}_{\xi}(\alpha) = 0$. Therefore $\forall x \in
A\[{f}_{\alpha}(x) = 0\]$ and ${f}_{\alpha}(y) = {f}_{\alpha}({x}_{\zeta}) =
{g}_{\zeta}(\alpha) = 1$, as required.
\end{proof}

\begin{Claim} \label{claim:indepcoding}
Suppose that there is a sequence $\seq{E}{\xi}{< }{{\kappa}^{+}}$ such that
the following conditions hold:
\begin{enumerate}
\item $\forall \xi < {\kappa}^{+}\[{E}_{\xi} \subseteq \kappa\]$;
\item for any countable $B \subseteq {\kappa}^{+}$ and any $\zeta \in
    {\kappa}^{+} \setminus B$, there is a finite set $F \subseteq \kappa$
    such that $\forall \xi \in B\[{E}_{\xi} \cap F \neq {E}_{\zeta} \cap
    F\]$.
\end{enumerate}
Then there is a sequence of functions $\seq{g}{\xi}{<}{{\kappa}^{+}}$ as in
Claim~\ref{claim:reading}.
\end{Claim}

\begin{proof}
Let $L = \{ \pr{s}{H}: s \in {\[\kappa\]}^{< \omega} \wedge H \subseteq
\Pset(s) \}$. Clearly $\lc L \rc = \kappa$. So it suffices to find a sequence
$\seq{h}{\xi}{<}{{\kappa}^{+}}$ of functions ${h}_{\xi}: L \rightarrow 2$
satisfying the properties of the sequence $\seq{g}{\xi}{<}{{\kappa}^{+}}$
from Claim~\ref{claim:reading}. Fix $\xi < {\kappa}^{+}$ and define
${h}_{\xi}: L \rightarrow 2$ by stipulating, for all $\pr{s}{H} \in L$, that
${h}_{\xi}(\pr{s}{H}) = 1$ if and only if ${E}_{\xi} \cap s \in H$. Let $B
\subseteq {\kappa}^{+}$ be countable and suppose $\zeta \in {\kappa}^{+}
\setminus B$. By~(2), there is a finite set $F \subseteq \kappa$ such that
$\forall \xi \in B\[{E}_{\xi} \cap F \neq {E}_{\zeta} \cap F\]$. Let $H = \{F
\cap {E}_{\zeta}\}$. Then $F \in {\[\kappa\]}^{< \omega}$ and $H \subseteq
\Pset(F)$. Therefore $\pr{F}{H} \in L$. Since ${E}_{\zeta} \cap F \in H$,
${h}_{\zeta}(\pr{F}{H}) = 1$. On the other hand, for any $\xi \in B$,
${E}_{\xi} \cap F \notin H$ because ${E}_{\xi} \cap F \neq {E}_{\zeta} \cap
F$. Hence ${h}_{\xi}(\pr{F}{H}) = 0$. So we have $\forall \xi \in
B\[{h}_{\xi}(\pr{F}{H}) = 0\]$, as required.
\end{proof}

\begin{Claim} \label{claim:a.d.}
There is a sequence $\seq{E}{\xi}{<}{{\kappa}^{+}}$ satisfying~(1) and~(2) of
Claim~\ref{claim:indepcoding}.
\end{Claim}

\begin{proof}
By well-known arguments (see, e.g., Kunen \cite[Theorem 1.2]{Kunen}), there
is a sequence $\seq{E}{\xi}{<}{{\kappa}^{+}}$, which is an almost disjoint
family of subsets of~$\kappa$. In other words, $\forall \xi <
{\kappa}^{+}\[{E}_{\xi} \in {\[\kappa\]}^{\kappa}\]$, and $\forall \xi, \zeta
< {\kappa}^{+}\[\xi \neq \zeta \implies \lc {E}_{\xi} \cap {E}_{\zeta} \rc <
\kappa \]$. The standard argument for the existence of such an almost
disjoint family (like the proof of \cite[Theorem 1.2]{Kunen}) uses the
regularity of $\kappa$. Suppose $B \subseteq {\kappa}^{+}$ is countable and
that $\zeta \in {\kappa}^{+} \setminus B$. Then $\lc {E}_{\zeta} \cap
{E}_{\xi} \rc < \kappa$ for all $\xi \in B$, and since $\cf(\kappa) >
\omega$,
$\lc {\bigcup}_{\xi \in B}{\left( {E}_{\zeta} \cap {E}_{\xi} \right)} \rc <
\kappa$. So choose $\alpha \in {E}_{\zeta} \setminus \left({\bigcup}_{\xi \in
B}{{E}_{\xi}} \right)$. Then $F = \{\alpha\}$ is a finite subset of~$\kappa$
and for any $\xi \in B$, $F \cap {E}_{\xi} = \emptyset$, while $F \cap
{E}_{\zeta} = \{\alpha\}$. So~$F$ is as required.
\end{proof}

\noindent
Now Claim~\ref{claim:a.d.} says that the hypotheses of
Claim~\ref{claim:indepcoding} are satisfied. So there is a sequence
$\seq{g}{\xi}{<}{{\kappa}^{+}}$ satisfying the hypotheses of
Claim~\ref{claim:reading}. Therefore, there is a sequence
$\seq{f}{\alpha}{<}{\kappa}$ as in the statement of the lemma.
\end{proof}

\begin{lem}\label{lem:linear}
Fix~$\kappa$ and $\PPP = (P,\<)$ as in Theorem~\ref{poset-dim},
$\{f_\alpha\}_{\alpha<\kappa}$ as in Lemma~\ref{lem:func}, and
$\alpha<\kappa$. Then there is a linearization~$<_\alpha$ of~$\<$ such that
for any $x, y \in P$, if $f_\alpha(u)=1$ for some $u \in P$ with $u\preceq x$
and $f_\alpha(z)=0$ for all $z \in P$ with $z\preceq y$, then $y <_\alpha x$.
\end{lem}

\begin{proof}
We first define a relation~$<^\prime_\alpha$ on $P\times P$ by setting $y
<^\prime_\alpha x$ iff $f_\alpha(u)=1$ for some $u\in P$ with $u\preceq x$
and $f_\alpha(z)=0$ for all $z \in P$ with $z \preceq y$. Note
that~$<^\prime_\alpha$ is irreflexive. We now only need to show that the
transitive closure of $\< \cup <^\prime_\alpha$ on~$P$ is irreflexive, since
any linearization~$<_\alpha$ of this transitive closure will be as desired.

So, for the sake of a contradiction, define $x <^{\prime\prime} y$ iff $x \<
y$ or $x <^\prime _\alpha y$, and assume that there is a finite sequence $x_0
<^{\prime\prime} x_1 <^{\prime\prime} \dots <^{\prime\prime} x_n = x_0$.
Since~$\<$ is irreflexive, we may assume, by a cyclic permutation of
the~$x_i$, that $x_0 <^\prime_\alpha x_1$. Furthermore, if $x_i
<^\prime_\alpha x_{i+1} \< x_{i+2}$, then $x_i <^\prime_\alpha x_{i+2}$ by
the definition of $<^\prime_\alpha$, and we can delete~$x_{i+1}$ from this
chain. Since $<^\prime_\alpha$ is irreflexive, we have $x_0<^\prime_\alpha
x_1<^\prime_\alpha x_2$. So we have that on the one hand, $f_\alpha(u)=1$ for
some $u\in P$ with $u\preceq x_1$ by the relation $x_0<^\prime_\alpha x_1$,
and on the other hand, $f_\alpha(u)=0$ for all $u\in P$ with $u\preceq x_1$
by the relation $x_1<^\prime_\alpha x_2$. This is a contradiction.
\end{proof}

\begin{proof}[Proof of Theorem~\ref{poset-dim}]
Assume that $x \not\preceq y$.
Let $A = \{z \in P: z \preceq y\}$, which is a
countable set, so apply Lemma~\ref{lem:func} to find a function~$f_\alpha$
such that $f_\alpha''A = \{0\}$ and $f_\alpha(x)=1$, and then apply
Lemma~\ref{lem:linear} with this function to obtain a linearization~$<_\alpha$
of~$\<$ with $y <_\alpha x$.

Thus, the at most~$\kappa$ many linearizations~$<_\alpha$ witness that the
dimension of~$\PPP$ is at most~$\kappa$.
\end{proof}

\noindent
Note that the proof of Theorem~\ref{poset-dim} really shows that if~$\lambda$
is some uncountable cardinal such that $\inva(\lambda) \geq {\kappa}^{+}$,
then the dimension of $(P, <)$ is at most $\lambda$, where $\inva(\lambda)$
is the almost disjointness number at $\lambda$.

Next, we consider the case when~$\kappa$ is possibly singular. We are able to
prove Theorem~\ref{poset-dim} also in that case if $\cf(\kappa) > \omega$.
The proof is almost the same.

\begin{thm} \label{c}
Suppose~$\kappa$ is any cardinal such that $\cf(\kappa) > \omega$ and $\PPP =
(P, \<)$ is any locally countable partial order of size~${\kappa}^{+}$.
Then~$\PPP$ has dimension at most~$\kappa$.
\end{thm}

\begin{proof}
The only place in the proof of Theorem~\ref{poset-dim} where the regularity
of~$\kappa$ is used is in Claim~\ref{claim:a.d.}. All of the other lemmas and
claims needed for the proof of Theorem~\ref{poset-dim} go through for any
uncountable cardinal~$\kappa$. So we assume that~$\kappa$ is a singular
cardinal with $\cf(\kappa) > \omega$. To establish the analogue of Claim
\ref{claim:a.d.}, we use a basic fact from PCF theory. Let $\mu =
\cf(\kappa)$. By hypothesis,~$\mu$ is an uncountable regular cardinal. By
\cite[Theorem 2.23]{handbookpcf}, there exist sequences
$\seq{\lambda}{\alpha}{<}{\mu}$ and $\seq{f}{\xi}{<}{{\kappa}^{+}}$
satisfying the following conditions:
\begin{enumerate}
\item for each $\alpha < \mu$, ${\lambda}_{\alpha} < \kappa$;
\item for each $\xi < {\kappa}^{+}$, ${f}_{\xi}$ is a function,
    $\dom({f}_{\xi}) = \mu$, and $\forall \alpha < \mu\[{f}_{\xi}(\alpha)
    \in {\lambda}_{\alpha}\]$;
\item for all $\xi < \zeta < {\kappa}^{+}$, ${Z}_{\xi,\zeta} \{ \alpha <
    \mu: {f}_{\xi}(\alpha) \geq {f}_{\zeta}(\alpha) \}$ is bounded
    in~$\mu$.
\end{enumerate}
Note that for each $\xi < {\kappa}^{+}$, ${f}_{\xi} \subseteq \mu \times
\kappa$. Furthermore $\lc {f}_{\xi} \rc = \mu$, and if $\zeta \neq \xi$, then
$\lc {f}_{\zeta} \cap {f}_{\xi} \rc \leq \lc \{\alpha < \mu:
{f}_{\zeta}(\alpha) = {f}_{\xi}(\alpha)\} \rc < \mu$. Therefore if $B
\subseteq {\kappa}^{+}$ is countable and $\zeta \in {\kappa}^{+} \setminus
B$, then $\lc {\bigcup}_{\xi \in B}{{f}_{\xi} \cap {f}_{\zeta}} \rc < \mu$
because $\mu > \omega$ is a regular cardinal. So we may find
$\pr{\alpha}{\delta} \in {f}_{\zeta} \setminus \left( {\bigcup}_{\xi \in
B}{{f}_{\xi} \cap {f}_{\zeta}} \right) $. As in the proof of
Claim~\ref{claim:a.d.}, let $F = \{\pr{\alpha}{\delta}\}$. $F$ is a finite
subset of $\mu \times \kappa$, and ${f}_{\zeta} \cap F =
\{\pr{\alpha}{\delta}\}$, while $\forall \xi \in B\[{f}_{\xi} \cap F =
\emptyset\]$. Thus the sequence $\seq{f}{\xi}{<}{{\kappa}^{+}}$ satisfies the
analogue of conditions~(1) and~(2) of Claim~\ref{claim:indepcoding}. Since
$\lc \mu \times \kappa \rc = \kappa$, we have shown that the hypotheses of
Claim~\ref{claim:indepcoding} are satisfied. Now the rest of the proof
follows the proof of Theorem~\ref{poset-dim}.
\end{proof}

\noindent
The following corollary is immediate from Theorems~\ref{poset-dim}
and~\ref{c}.

\begin{cor}
If the order dimension of any locally countable partial order
is~${2}^{{\aleph}_{0}}$, then either~$\CH$ holds, or ${2}^{{\aleph}_{0}}$ is
a limit cardinal, or~${2}^{{\aleph}_{0}}$ is the successor of a singular
cardinal of countable cofinality.
\end{cor}

\noindent
Our conjecture is that the cases besides~$\CH$ are also realized. In other
words, we think that there are models where ${\mathfrak{dim}}_{T} =
{2}^{{\aleph}_{0}} = {\aleph}_{{\omega}_{1}}$, ${\mathfrak{dim}}_{T} =
{2}^{{\aleph}_{0}} = {\aleph}_{\omega+1}$, and ${\mathfrak{dim}}_{T} =
{2}^{{\aleph}_{0}}$ where~${2}^{{\aleph}_{0}}$ is weakly inaccessible.

\section{\texorpdfstring{The dimension of some degree structures\\
from computability theory}%
{The dimension of some degree structures from computability theory}}%
\label{sec:degstr}

\noindent
In this section, we state some results on the dimension of three degree
structures from computability theory, the Turing degrees, the Medvedev
degrees and the Muchnik degrees.

We start with the Turing degrees since the original motivation for our
investigation was determining the dimension of the Turing degrees under the
partial ordering, for which we obtain two partial results.

\begin{defn}
For a partial order $\PPP = (P,\<)$ and a subset~$S$ of~$P$, we say that~$S$
is a \emph{strongly independent} antichain if for any subset~$T$ of~$S$ with
$|T| < |S|$ and for any $x\in S\setminus T$, there is an upper bound $y\in P$
of~$T$ with $y \not\succeq x$.
\end{defn}

\begin{lem}\label{dim-lemma}
Let $\PPP = (P,\<)$ be a partial order with a strongly independent
antichain~$S$. Then the dimension of~$\PPP$ is at least~$|S|$.
\end{lem}

\begin{proof}
We provide a proof by contradiction. Suppose that the dimension~$\rho$
of~$\PPP$ is less than~$|S|$. Let $\hat{S}\subset S$ satisfy that
$|\hat{S}|>\rho$ is a successor cardinal and let $\kappa=|\hat{S}|$. Choose
linear extensions $\{<_\alpha\}_{\alpha<\rho}$ of~$\<$ whose intersection
is~$\<$.

Fix $x\in\hat{S}$. Let $\{T_\alpha\}_{\alpha<\kappa}$ be any increasing
sequence of subsets of~$\hat{S}$ of cardinality $<|\hat{S}|$ such that
$\bigcup_{\alpha<\kappa}T_\alpha=\hat{S}\setminus\{x\}$. By the strong
independence, we can find a sequence $\{y_\alpha\}_{\alpha<\kappa}$ of upper
bounds of~$T_\alpha$'s such that for each $\alpha<\kappa$, $x \not\preceq
y_\alpha$, which means that there exists $\beta<\rho$ such that $y_\alpha
<_\beta x$ by the choice of $\{<_\beta\}_{\beta<\rho}$. Since $\rho<\kappa$
and~$\kappa$ is regular (or finite),
there must exist a fixed $\beta<\rho$ such that $y_\alpha <_\beta x$ holds
for unboundedly many $\alpha<\kappa$. By the choice of
$\{T_\alpha\}_{\alpha<\kappa}$, every element of~$\hat{S}$ distinct from~$x$
is in almost all $\{T_\alpha\}_{\alpha<\kappa}$, and therefore, $y <_\beta x$
must hold for each $y\in\hat{S}\setminus\{x\}$. Hence we conclude that for
any $x\in \hat{S}$, there exists $\beta<\rho$ such that $y <_\beta x$.

Since again $\rho<\kappa=|\hat{S}|$, there must exist a common $\beta<\rho$
and distinct $x_0,x_1\in\hat{S}$ such that $y <_\beta x_i$ holds for each
$y\in \hat{S}\setminus\{x_i\}$ and $i\in\{0,1\}$. But this gives us $x_0
<_\beta x_1 <_\beta x_0$, and hence $x_0 <_\beta x_0$, a contradiction.
\end{proof}

\noindent
We can now state two partial results about the dimension of the Turing
degrees as a partial order:

\begin{prop}[Higuchi] \label{prop:higuchi}
The dimension of the Turing degrees is uncountable.
\end{prop}

\begin{proof}
By Sacks~\cite{Sacks}, every locally countable partial order of
cardinality~$\aleph_1$ is embeddable into the Turing degrees. Thus it is
enough to find such a partial order whose dimension is at least~$\aleph_1$.
Let us consider the suborder $(P, \subsetneq)$ of $\Pset({\aleph}_{1})$ under
set inclusion whose underlying set is
$$
P=\{\{\alpha\}:\alpha<\omega_1\}\cup
  \{\{\gamma:\gamma<\beta\}\setminus\{\alpha\}:\alpha,\beta<\omega_1\}.
$$
It is easy to see that $\{\{\alpha\}:\alpha<\omega_1\}$ is a strongly
independent antichain of cardinality~$\aleph_1$ in the partial order $(P,
\subsetneq)$, and therefore, the dimension of $(P, \subsetneq)$ is at
least~$\aleph_1$ by Lemma~\ref{dim-lemma}.
\end{proof}

\begin{thm}
It is consistent with ZFC that the dimension of the Turing degrees is
strictly less than the continuum. More precisely, the dimension of the Turing
degrees can be continuum only if either~$\CH$ holds, or ${2}^{{\aleph}_{0}}$
is a limit cardinal, or~${2}^{{\aleph}_{0}}$ is the successor of a singular
cardinal of countable cofinality.
\end{thm}

\begin{proof}
The first part is a direct corollary of Theorems~\ref{poset-dim}: Work in a
model in which $2^{\aleph_0} = \aleph_2$ and apply the theorem with $\kappa =
\aleph_1$. The second part combines Theorems~\ref{poset-dim} and~\ref{c}.
\end{proof}

\noindent
We now turn our attention to the Medvedev and the Muchnik degrees. By a
cardinality argument, the dimension of both can be at
most~$2^{2^{\aleph_0}}$.

\begin{thm}[Pouzet 1969]
Let $\PPP = (P,\<)$ be a partial order. Then the dimension of
$(\IniSeg(\PPP),\subset))$ is the chain covering number of~$\PPP$, where
$\IniSeg(\PPP)$ is the set of initial segments of~$\PPP$ and the \emph{chain
covering number} of~$\PPP$ is the least cardinal~$\kappa$ such that there
exists a set~$\CCC$ of chains of~$\PPP$ with $|C| = \kappa$ and $\bigcup \CCC
= P$.
\end{thm}

\noindent
It is known that the Muchnik degrees are isomorphic to the set of all final
segments of the Turing degrees ordered by~$\supset$. Note that the dimension
of a partial order does not change if we reverse the order. Thus the
dimension of the Muchnik degrees is the chain covering number of the Turing
degrees, which is~$2^{\aleph_0}$ since there are at most~$2^{\aleph_0}$ many
Turing degrees and the Turing degrees contain an antichain of
size~$2^{\aleph_0}$. Since the Muchnik degrees can be seen a suborder of the
Medvedev degrees, the dimension of the latter is at least~$2^{\aleph_0}$, and
by a cardinality argument at most~$2^{2^{\aleph_0}}$.

We can thus determine the dimension of the Muchnik degrees in ZFC but leave
open the following questions:

\begin{ques}
\begin{enumerate}
\item Does ZFC determine the dimension of the Turing degrees?
\item Does ZFC determine the dimension of the Medvedev degrees?
\end{enumerate}
\end{ques}

\section{Some examples}

\noindent
One of the goals of our paper was to find out whether the order dimension of
the structure of Turing degrees is~$\aleph_1$ in all models of ZFC. More
generally, we can pose the following

\begin{ques}
Is there a partially ordered set $(F,<)$ such that, in all models of ZFC, the
following holds?
\begin{enumerate}
\item The order~$<$ is locally countable;
\item for every at most countable subset $G \subseteq F$ there is an upper
    bound~$x$ of~$G$, i.e., $y \le x$ for all $y \in F$;
\item the cardinality of~$F$ is~$2^{\aleph_0}$; and
\item the order dimension of $(F,<)$ is~$\aleph_1$.
\end{enumerate}
\end{ques}

\noindent
Each of the examples below will satisfy three of the properties but the
fourth at most partially. Note that additional set-theoretical assumptions
like $2^{\aleph_0} \in \{\aleph_1,\aleph_2\}$ might make the examples have
all the desired properties.



\begin{example}\label{ex1}
Let~$F$ be the set of all hereditarily countable sets and let~$<$ be the
transitive closure of the element-relation. As every set bounds only
countably many other sets in a hereditary way, the partial order is locally
at most countable, and it is well-founded. Furthermore, every at most
countable ordinal~$\alpha$ can be identified with the at most countable set
$\{\beta: \beta < \alpha\}$, and these sets are hereditarily countable. The
set $\{\{\{\alpha\}\}: \alpha$ is an at most countable ordinal$\}$ is an
antichain of size~$\aleph_1$ consisting of hereditarily countable sets; as
each of its at most countable subsets is hereditarily countable, these sets
witness that the antichain is indeed a strong antichain and that therefore
the order dimension is at least $\aleph_1$. It is known that the set of
hereditarily countable sets has cardinality~$2^{\aleph_0}$.
\end{example}

\begin{example}\label{ex2}
Let~$F$ be the set of all functions~$f$ from an ordinal $\alpha < \omega_1$
into~$\omega_1$; order~$F$ by letting $f < g$ iff there are ordinals
$\alpha,\beta \in \dom(g)$ such that for all $\gamma \in \dom(f)$,
$f(\gamma) = g(\alpha+1+\gamma)$, $g(\alpha) = g(\beta)$, $f(\gamma) <
g(\alpha)$ and $\alpha+1+\dom(f) = \beta$. It is easy to see that~$<$ is
transitive and locally countable. Furthermore, one can easily see that for at
most countably many functions $f_0, f_1, \ldots$, there is a common upper
bound~$g$ by choosing an ordinal~$\alpha$ strictly larger than all ordinals
occurring in the~$f_k$, considering an $\omega$-power $\omega^\gamma$ larger
than the domains of all~$f_k$, letting the domain of $h$ be
$\omega^{\gamma+1}$, and setting $h(\omega^\gamma \cdot k+1+\gamma) =
f_k(\gamma)$ for all $\gamma \in \dom(f_k)$ and $h(\delta) = \alpha$ for
all~$\delta$ in the domain of~$h$ where~$h$ is not yet defined. So every at
most countable set~$G$ of members of~$F$ has a common upper bound, which
strictly bounds from above exactly the members of~$G$ and those members
of~$F$ which are below a member of~$G$. The cardinality of~$F$
is~$2^{\aleph_0}$. Furthermore, the set of all functions with domain~$\{0\}$
forms a strong antichain and therefore, the order dimension is at
least~$\aleph_1$.
\end{example}

\begin{example}\label{ex3}
Let~$F$ be the set of all subsets of~${\mathbb R} \times {\omega_1}$ such
that~$F$ is the union of finitely many sets of the form $A_{x,y,z} = \{x\}
\times \{u \in M: y \leq u \leq z\}$ and $B_{x,y,z} = \{x\} \times \{u \in M:
y \leq u < z\}$, and order this set by set inclusion. The set $(F,<)$
satisfies all four conditions except for the second, which is weakened to the
existence of common upper bounds of finitely many elements.

One can see from the definition that every set of the form $A_{x,y,z}$ or
$B_{x,y,z}$ has at most countably many subsets of this form in~$F$;
furthermore each member of~$F$ is countable. Thus this is a locally countable
partially ordered set.

Furthermore, as the finite union of any members of~$F$ is again a member of
$F$, one has also that finitely many subsets have an upper bound. However,
this does not extend to all countable subsets of~$F$.

The cardinality of~$F$ is~$2^{\aleph_0}$. The lower bound is seen by looking
at all sets $A_{x,0,0}$ with $x \in \mathbb R$. The upper bound stems from
the fact that there are~$2^{\aleph_0}$ many countable subsets of~$\mathbb R$.

Furthermore, let $p,q$ be rational numbers with $p<q$ and $y \in {\omega_1}$.
For each $p,q,y$, one defines a linear order $<_{p,q,y}$ as a linear
extension of the partial order defined by $A <_{p,q,y} B$ iff~$B$ has more
elements than~$A$ of the form $(x,y)$ with $p \leq x \leq q$; note that each
set has only finitely many such elements. Now if $B \not\subseteq A$ then
there is an $(x,y) \in B \setminus A$ and there are rationals $p,q$ such that
$x$ is the unique real number~$z$ with $p \leq z \leq q$ and $(z,y) \in A
\cup B$. It follows that $A <_{p,q,y} B$. So~$<$ is the intersection of all
$<_{p,q,y}$ with $p,q \in \mathbb Q$ and $y \in \omega_1$ and $p<q$. It
follows that the order dimension of $(F,<)$ is at most~$\aleph_1$.

Now consider the set~$C$ of all $A_{0,y,y}$, which are all singletons. The
set $C$ forms a strong antichain, as for every $A_{0,x,x}$ and every at most
countable set~$D$ of sets of the form $A_{0,y,y}$ with $y \neq x$, there is
an upper bound~$z$ of all these~$y$. For this upper bound~$z$, now consider
the set $E = B_{0,0,x} \cup A_{0,x+1,z}$, which is a superset of all
$A_{0,y,y} \in D$ but not a superset of $A_{0,x,x}$. Now by
Lemma~\ref{dim-lemma}, $(F,<)$ has order dimension at least~$\aleph_1$, as
this is the cardinality of~$C$.
\end{example}


\begin{example}\label{ex4}
Given~$F$ as in Example~\ref{ex3}, the subset $G = \{A \in F: A \subseteq
\{0\} \times {\omega_1}\}$ satisfies the first, second and last property, but
differs from the third in all models of set theory where $\aleph_1 \neq
2^{\aleph_0}$.
\end{example}


\begin{example}\label{ex5}
The set~$F$ of all countable subsets of~$\omega_1$ with the order of
inclusion satisfies the property that every countable subset of~$F$ has an
upper bound in~$F$ and the cardinality is~$2^{\aleph_0}$. Furthermore,
$(F,<)$ has order-dimension~$\aleph_1$.
\end{example}


\end{document}